\documentclass[11pt]{amsart}

\usepackage{amsmath}
\usepackage{amssymb}
\usepackage{graphicx}
\usepackage[
  citestyle=numeric,
  bibstyle=numeric
]{biblatex}

\newtheorem{thm}{Theorem}[section]
\newtheorem*{thm*}{Theorem}
\newtheorem{prop}[thm]{Proposition}
\newtheorem{lem}[thm]{Lemma}
\newtheorem*{lem*}{Lemma}
\newtheorem*{cor*}{Corollary}
\theoremstyle{definition}
\newtheorem{def_n}[thm]{Definition}
\newtheorem{cor}{Corollary}[section]
\newtheorem{ex}[thm]{Example}

\newtheorem{rem}[thm]{Remark}

\numberwithin{equation}{section}

\newcommand{\Hcal}{\mathcal{H}} 
\newcommand{\N}{\mathbb{N}}

\title{Sequences that do frame reconstruction}
\author{Chad Berner}

\begin{document}

\begin{abstract}
Frames allow all elements of a Hilbert space to be reconstructed by inner product data in a stable manner. Recently, there is interest in relaxing the definition of frames to understand the implications for stable signal recovery. In this paper, we relax the definition of a frame by allowing the operator in the frame decomposition formula to not be invertible. We provide a complete classification of sequences that allow this decomposition via a type of frame operator. Additionally, we provide several examples of sequences that allow this reconstruction property that are not frames and illustrate in which ways they fail to be frames. Furthermore, we provide a Paley-Wiener type stability result for sequences that do this frame-like reconstruction, which is also stable under the non-frame property. Finally, we classify certain Schauder bases—such as unconditional and exponential bases—that satisfy this relaxed frame reconstruction condition.
\end{abstract}

\maketitle

\section{Introduction}
Given a sequence of vectors $\{f_n\}$ in a Hilbert space $\Hcal$, one can ask if $\{f_n\}$ can recover any signal $f\in \Hcal$ via its inner product data $\{\langle f, f_n\rangle\}$ in a stable way. Frame theory provides a strong approach to answer this question, and its origin began with Duffin and Schaeffer \cite{Duffin1952Class}. Specifically, if $\{f_n\}_{n=0}^{\infty}\subseteq \Hcal$ is a frame, then there is a positive definite invertible $S\in B(\Hcal)$ called the \textbf{frame operator} that allows stable reconstruction of all signals in $\Hcal$. Formally, the most important aspect of frames is the frame decomposition, which holds for all $f\in \Hcal$:
\begin{equation}\label{framedecomp}
f=\sum_{n=0}^{\infty} \langle f, S^{-1}f_n\rangle f_n=\sum_{n=0}^{\infty}\langle S^{-1}f,f_n\rangle f_n=\sum_{n=0}^{\infty} \langle f, f_n\rangle S^{-1}f_n.    
\end{equation}
Furthermore, these series converge unconditionally.

An equivalent definition of a frame, which we will prove shortly, is the existence of an invertible $B\in B(\Hcal)$ such that
for all $f\in \Hcal$:
$$f=\sum_{n=0}^{\infty} \langle f, Bf_n\rangle f_n.$$ However, if we only require that there is a $B\in B(\Hcal)$ such that
for all $f\in \Hcal$:
\begin{equation}\label{doesframereconintro}
f=\sum_{n=0}^{\infty} \langle f, Bf_n\rangle f_n,\end{equation} then $\{f_n\}_{n=0}^{\infty}$ need not be a frame, although $B$ must be injective. Intuitively, the property described in equation \eqref{doesframereconintro} means $\{f_n\}_{n=0}^{\infty}$ is as close to being a frame as possible without being a frame. In this case, we say $\{f_n\}_{n=0}^{\infty}$ does \textbf{frame reconstruction}. We will also show that a sequence that does frame reconstruction still admits a frame decomposition via a bounded operator in the sense of equation \eqref{framedecomp}; therefore, sequences that do frame reconstruction still recover signals in a bounded and stable way. Our analysis will show that sequences that do frame reconstruction are quite different from frames even though they share many of their desirable properties.

There is considerable interest in expanding the definition of frames and understanding the implications for stable signal recovery. For example, Christensen provided results on frame-like decompositions involving unbounded operators \cite{Christensen1995Frames}. Additionally, Li and Ogawa relaxed the definition of frames in \cite{Li2001Pseudo-duals}, studying the implications when convergence of frames series is merely in the weak topology. Furthermore, frame-like series built from the Kaczmarz algorithm have been investigated, see \cite{Haller2005Kaczmarz},\cite{Herr2017Fourier},\cite{Kwapien2001Kaczmarz}. Moreover, Balazs, Antoine, and Grybos along with many others have studied weighted frames, that is, sequences that can be converted to frames via positive scalars and their applications to Fourier expansions, see \cite{Balazs2010Weighted},\cite{Dutkay2016Weighted}. In this work, we will establish some connections between sequences that do frame reconstruction, weighted frames, and sequences that arise from the Kaczmarz algorithm. 

The primary purpose of this paper is to establish the properties that sequences that do frame reconstruction have in relation to frames and provide explicit classification and examples of such sequences. The secondary purpose of this paper is to relate sequences that do frame reconstruction to other generalizations of frames and to classify certain Schauder bases that do frame reconstruction.

The outline of this paper is as follows: In section \ref{s3}, utilizing concrete examples, we establish both the shared properties of frames and sequences that do frame reconstruction, as well as the properties that distinguish them. We also give a classification of sequences that do frame reconstruction using a frame operator. Then, in section \ref{s4}, we present two examples of sequences that do not do frame reconstruction and identify the specific properties of frame reconstruction they fail to satisfy. Furthermore, in section \ref{s5}, we provide a Paley-Wiener type stability result for sequences that do frame reconstruction, which is also stable under the non-frame property. Additionally, in section \ref{s6}, we relate sequences that do frame reconstruction with a special class of weighted frames. Then, in section \ref{s7}, we classify all unconditional Schauder bases do frame reconstruction. Finally, in section \ref{s8}, we classify the finite Borel measures on the torus for which the classical exponential system constitutes a Schauder basis that enables frame reconstruction.
\section{Preliminaries}
We begin by providing the definition of a frame:
\begin{def_n}
Let $\{f_{n}\}_{n}\subseteq \Hcal$ where $\Hcal$ is a Hilbert space.
\begin{enumerate}
\item If there exists $A>0$ such that
$$A\|f\|^{2}\leq \sum_{n}|\langle f, f_{n}\rangle|^{2}$$
for all $f\in \Hcal$, then $\{f_{n}\}_{n}$ is called a \textbf{lower semi-frame}.
\item If there exists $B>0$ such that
$$\sum_{n}|\langle f, f_{n}\rangle|^{2}\leq B\|f\|^{2}$$
for all $f\in \Hcal$, then $\{f_{n}\}_{n}$ is called a \textbf{Bessel sequence}.
\item If there exists $A,B>0$ such that
    $$A\|f\|^{2}\leq \sum_{n}|\langle f, f_{n}\rangle|^{2}\leq B\|f\|^{2}$$
    for all $f\in \Hcal$, then $\{f_{n}\}_{n}$ is called a \textbf{frame}.
\item If
$$\|f\|^{2}=\sum_{n}|\langle f, f_{n}\rangle|^{2}$$
for all $f\in \Hcal$, then $\{f_{n}\}_{n}$ is called a \textbf{Parseval frame}.
\item If $\{f_{n}\}$ is a frame, and
$$\sum_{n}c_{n}f_{n}=0\implies \{c_{n}\}=0,$$
whenever $\{c_{n}\}\in \ell^{2}$,
then $\{f_{n}\}_{n}$ is called a \textbf{Riesz basis}.
\end{enumerate}
Furthermore, if $\{f_{n}\}_{n}$ is a Bessel sequence in Hilbert space $\Hcal$,
the map on $\Hcal$: $f\to \{\langle f, f_{n}\rangle \}_{n}$ is called the \textbf{analysis operator}, its adjoint is called the \textbf{synthesis operator}, and 
we denote $S: \Hcal\to \Hcal$ as the positive operator such that
$$S(f)=\sum_{n}\langle f, f_{n}\rangle f_{n},$$ which is called the \textbf{frame operator}.
Additionally, if $\{f_{n}\}_{n}$ is a frame, then its associated frame operator is invertible, and $\{S^{-1}f_{n}\}_{n}$ is a frame such that
$$f=\sum_{n}\langle f, S^{-1}f_{n}\rangle f_{n}=\sum_{n}\langle f, f_{n}\rangle S^{-1}f_{n}$$ with convergence in norm for all $f\in \Hcal$.

Additionally, if $\{f_{n}\}$ and $\{g_{n}\}$ are frames (Bessel sequences) such that for all $f\in \Hcal$,
$$f=\sum_{n}\langle f, f_{n}\rangle g_{n},$$ then $\{f_{n}\}$ and $\{g_{n}\}$ are called \textbf{dual frames}. In particular, many frames have multiple dual frames. Additionally, because the synthesis operator and analysis operator associated with both sequences are bounded, by an adjoint calculation, 
$$f=\sum_{n}\langle f,g_{n}\rangle f_{n}$$ also holds for all $f\in \Hcal$. Finally, the norm convergence of frame series is also unconditional.
\end{def_n}

Furthermore, we define Schauder bases:
\begin{def_n}
Let $\{f_{n}\}_{n}\subseteq \Hcal$ where $\Hcal$ is a Hilbert space. $\{f_n\}_{n}$ is called a \textbf{Schauder basis} if there is a $\{g_{n}\}_{n}\subseteq \Hcal$ such that
$$\langle f_n,g_k\rangle=\delta_{nk}$$ and
\begin{equation}\label{schauderdef}
f=\sum_{n}\langle f,g_n\rangle f_n
\end{equation}
with convergence in norm for all $f\in \Hcal$.
If the convergence in equation \eqref{schauderdef} is unconditional for all $f\in \Hcal$, then $\{f_n\}_n$ is called an \textbf{unconditional Schauder basis}.
\end{def_n}
\begin{rem}
A Schauder basis may not be a frame, and a frame may not be a Schauder basis. However, Riesz bases are exactly the sequences that are both frames and Schauder bases.
\end{rem}

Additionally, we will make good use of the following results about Riesz bases and unconditional convergence where proofs can be found in \cite{Gohberg1969Theory}:
\begin{thm*}[Gohberg]
Let $\Hcal$ be a Hilbert space. $\{f_n\}\subseteq \Hcal$
is a Riesz basis if and only if it
is an unconditional Schauder basis and
$$0 < \inf_{n}
\|f_n\| \leq \sup_{n}
\|f_n\| < \infty.$$
\end{thm*}

\begin{thm*}[Gohberg]
Let $\Hcal$ be a Hilbert space and $\{f_n\}_{n\in I}\subseteq \Hcal$. If $\sum_{n\in I}f_n$ converges unconditionally, then
$$\sum_{n\in I}\|f_n\|^{2}<\infty.$$
\end{thm*}

We will also utilize the following result that provides a condition for when operators have a bounded right inverse, see \cite{Beutler1976Operator} for a proof.
\begin{thm*}[Beutler]
Let $\Hcal$ be a Hilbert space and suppose that $L: D(L)\subseteq \mathcal{H}\to H$ is linear, closed, surjective, and $\overline{D(L)}=\mathcal{H}$. Then there is $L^{\dagger}\in B(\mathcal{H})$ such that for all $f\in \Hcal$,
$$LL^{\dagger}f=f.$$
\end{thm*}

\section{Sequences that do frame reconstruction}\label{s3}
In this section, we make precise the definition of a sequence doing frame reconstruction. We will also provide examples of sequences that do frame reconstruction that are not frames as well as compare the properties of frames and sequences that do frame reconstruction. Finally, we close this section with a classification of sequences that do frame reconstruction using a frame operator.
\begin{def_n}
Let $\Hcal$ be a Hilbert space. A sequence \(\{f_n\}_{n=0}^\infty \subset \Hcal\) \textbf{does frame reconstruction with $B\in B(\Hcal)$} if for all \(f \in \Hcal\),
\[
\sum_{n=0}^\infty \langle f, Bf_n \rangle f_n = f,
\]
where the series converges in norm. 

Similarly, a sequence \(\{f_n\}_{n\in \mathbb{Z}} \subset \Hcal\) \textbf{does frame reconstruction with $B\in B(\Hcal)$} if for all \(f \in \Hcal\),
\[
\lim_{M\to \infty}\sum_{n=-M}^{M} \langle f, Bf_n \rangle f_n = f,
\]
where the series converges in norm.

Note that this definition does not assume anything about unconditional convergence, but if either series above converge unconditionally for all $f\in \Hcal$, then we say $\{f_{n}\}_{n\in I}$ \textbf{does unconditional frame reconstruction with $B$}.
\end{def_n}
First, we will show that a sequence doing frame reconstruction does not require the sequence to be a frame with following example:
\begin{ex}\label{firstnonframeex}
Consider the standard orthonormal basis \(\{e_n\}_{n=0}^\infty\) of \(\ell^2(\N)\). Define
\[
f_n := (n+1) e_n.
\]
Then \(\{f_n\}_{n=0}^\infty\) does unconditional frame reconstruction with $B$ where
$$B\sum_{n=0}^{\infty}a_{n}e_{n}=\sum_{n=0}^\infty \frac{a_n}{(n+1)^2} e_n.$$
Indeed, for any \(f \in \ell^2(\N)\),
\begin{equation}\label{dfrandnotfex1}
\sum_{n=0}^\infty \langle f, Bf_n \rangle f_n = \sum_{n=0}^\infty \left\langle f, \frac{e_n}{n+1} \right \rangle (n+1 )e_n=f.\end{equation}
Furthermore, a similar calculation shows that for any \(f \in \ell^2(\N)\),
\begin{equation}\label{dfrandnotfex2}
\sum_{n=0}^\infty \langle f, f_n \rangle Bf_n =f.\end{equation}
Additionally, the convergence in equation \eqref{dfrandnotfex1} and equation \eqref{dfrandnotfex2} is unconditional; therefore, in almost every way, $\{f_{n}\}$ does what one would want from a frame. However, it is easy to check that $\{f_{n}\}$ is not a frame as it fails to be a Bessel sequence.
\end{ex}
In the previous example, $\{f_{n}\}$ did unconditional frame reconstruction, but small changes to vectors $f\in \mathcal{H}$ may result in large changes in inner product data: $\{\langle f, f_n\rangle\}$ since $\{f_n\}$ is an unbounded sequence. However, now we provide an example of a bounded sequence that does unconditional frame reconstruction that is not a frame:
\begin{ex}\label{framereconex2}
Consider the standard orthonormal basis \(\{e_n\}_{n=0}^\infty\) of \(\ell^2(\N)\). Define $\{f_n\}_{n=0}^{\infty}$ to be the following sequence: 
$$\{e_{0}, e_1, e_1, e_2, e_2, e_2,\dots\}.$$
It is easy to see that $\{f_n\}$ is not a Bessel sequence since 
$$\sum_{n=0}^{\infty}|\langle e_{k},f_{n}\rangle|^{2}=k+1.$$
However, if we define $B\in B(\ell^{2}(\mathbb{N}))$ where
$$B\sum_{n=0}^{\infty}a_{n}e_{n}=\sum_{n=0}^\infty \frac{a_n}{n+1} e_n,$$ then $\{f_n\}$ does unconditional frame reconstruction with $B$, which follows from
$$\{e_{0}, \frac{e_1}{\sqrt{2}}, \frac{e_1}{\sqrt{2}}, \frac{e_2}{\sqrt{3}}, \frac{e_2}{\sqrt{3}},\frac{e_2}{\sqrt{3}},\dots\}$$ being a Parseval frame.
\end{ex}
We will see soon however, that a bounded minimally complete sequence that does unconditional frame reconstruction must be a frame. Now we show that if $\{f_{n}\}$ does frame reconstruction with $B$, then $B$ is unique. Therefore, we are justified in saying $\{f_n\}$ \textbf{does frame reconstruction} instead of $\{f_n\}$ \textbf{does frame reconstruction with $B$}.
\begin{prop}
Let $\Hcal$ be a Hilbert space and suppose $\{f_{n}\}_{n\in I}\subseteq \Hcal$ where $I$ is $\mathbb{N}$ or $\mathbb{Z}$ does frame reconstruction with $B$ and $V$. Then $$B=V.$$
\end{prop}
\begin{proof}
We will assume $I=\mathbb{N}$ since the proof when $I=\mathbb{Z}$ is similar.
We have for all $f\in \Hcal$,

\[
\sum_{n=0}^\infty \left\langle f, (B-V)f_n \right\rangle f_n =0.
\]
Therefore for all $f\in \Hcal$,

\[
\sum_{n=0}^\infty \left\langle (B - V)^{*} f, f_n  \right \rangle \left \langle f_n, (B - V)^{*} f  \right\rangle = 0.
\]
Now since $\{f_{n}\}_{n\in \mathbb{N}}$ is clearly complete, for all $f\in \Hcal$,
$$(B-V)^{*}f=0.$$
\end{proof}

To justify the name \textbf{does frame construction}, we show that if $\{f_{n}\}$ does frame reconstruction, it enjoys many of the same properties that frames do. In particular, sequences that do frame reconstruction enjoy a type of frame decomposition:

\begin{thm}\label{doesframerecprop}
Let $\Hcal$ be a Hilbert space and suppose $\{f_{n}\}_{n\in I}\subseteq \Hcal$ where $I$ is $\mathbb{N}$ or $\mathbb{Z}$ does frame reconstruction with $B$. Then 
\begin{enumerate}
    \item \(B\) is positive semi-definite, injective, and has dense range.
    \item For all $f\in \Hcal$, $$\sum_{n\in I}\langle f, f_n\rangle Bf_n=\sum_{n\in I}\langle Bf, f_n\rangle f_n=f$$ where the sum is done appropriately depending on if $I$ is $\mathbb{N}$ or $\mathbb{Z}$.
    \item For all $f\in \Hcal$, $$\sum_{n\in I}|\langle f, Bf_{n}\rangle|^{2} \leq  \|B\| \|f\|^{2}.$$
    \item For all $f\in \Hcal$,
    \[
    \frac{1}{\|B\|}\|f\|^2 \leq \sum_{n\in I} |\langle f, f_n \rangle|^2.
    \]
    \item $\{\sqrt{B}f_n\}_{n\in I}$ is a Parseval frame.
\end{enumerate}
\end{thm}
\begin{proof}
Again we will assume $I=\mathbb{N}$ since the proof when $I=\mathbb{Z}$ is similar.

$(1)$:

Consider for any $f\in \Hcal$ by the frame reconstruction property, $$\langle f,B^{*}f\rangle= \sum_{n=0}^{\infty}\langle f, Bf_n\rangle \langle f_n, B^{*}f\rangle \geq 0.$$
Furthermore, it is clear that $\{Bf_n\}$ is dense in $\mathcal{H}$ by the frame reconstruction property, so $(1)$ follows.

$(2)$:

By the principle of uniform boundedness, $R_{N}: \Hcal \to \Hcal$ where
$$R_{N}f=\sum_{n=0}^{N}\langle f, Bf_n\rangle f_n$$ are a uniformly bounded sequence of operators. Therefore,
$R_{N}^{*}$ are a uniformly bounded sequence of operators where
$$R_{N}^{*}f=\sum_{n=0}^{N}\langle f ,f_n\rangle Bf_n.$$ Furthermore, since $B$ is bounded and self-adjoint,
$$Bf=B\sum_{n=0}^{\infty}\langle f, Bf_n\rangle f_n=\sum_{n=0}^{\infty}\langle f, Bf_n\rangle Bf_n=\sum_{n=0}^{\infty}\langle Bf, f_n\rangle Bf_n$$
for all $f\in \Hcal.$ That is, $R_{N}^{*}$ converge to $I$ pointwise on $ImB$, which is dense in $\Hcal$. Therefore, $R_{N}^{*}$ must converge to $I$ pointwise on $\Hcal$.

$(3):$

This follows easily from the proof of $(1)$ and Cauchy–Schwarz inequality.

$(4):$

Note for any $f\in \Hcal$ by $(3)$,
$$\langle f, f\rangle=\sum_{n\in I}\langle f, Bf_n\rangle \langle f_n, f\rangle  \leq \sqrt{\|B\|}\|f\| \sqrt{\sum_{n\in I}|\langle f, f_{n}\rangle|^{2}}.$$

$(5):$

By applying the reconstruction formula and applying $\sqrt{B}$ to both sides, we have
$$\sqrt{B}f=\sum_{n=0}^{\infty}\left\langle \sqrt{B}f, \sqrt{B}f_n \right\rangle \sqrt{B}f_n$$
for all $f\in \mathcal{H}$. Therefore, 
$$\left \|\sqrt{B}f \right\|^{2}=\sum_{n=0}^{\infty}\left|\left\langle \sqrt{B}f, \sqrt{B}f_n \right\rangle \right|^{2}$$
for all $f\in \mathcal{H}$. Now by $(1)$, $\sqrt{B}$ has dense range. Suppose for the sake of contradiction and without lose of generality that for some $f\in \mathcal{H}$,
$$\sum_{n=0}^{\infty}\left|\left\langle f, \sqrt{B}f_n \right\rangle \right|^{2}>\|f\|^{2}.$$
Then for some $M$,
$$\sum_{n=0}^{M}\left|\left\langle f, \sqrt{B}f_n \right\rangle \right|^{2}>\|f\|^{2}.$$
Therefore, there is a $g\in \mathcal{H}$ such that
$$\sum_{n=0}^{M}\left|\left\langle \sqrt{B}g, \sqrt{B}f_n \right\rangle \right|^{2}> \left \|\sqrt{B}g \right\|^{2},$$ which is a contradiction. 
\end{proof}

\begin{cor}\label{almostframecor}
Let $\Hcal$ be a Hilbert space and suppose $\{f_{n}\}_{n\in I}\subseteq \Hcal$ where $I$ is $\mathbb{N}$ or $\mathbb{Z}$ does frame reconstruction with $B$. If \emph{any} of the following conditions hold:
\begin{enumerate}
\item $\{f_n\}_{n\in I}$ is a Bessel sequence,
\item \(B\) has closed range,
\item there exists $A>0$ such that for all $f\in \Hcal$,
    \[
    A\|f\|^2 \leq \sum_{n\in I} |\langle f, Bf_n \rangle|^2,
    \]
\item \(\inf_n \|B f_n\| > 0\) and $\{f_n\}_{n\in I}$ does unconditional frame reconstruction,
\end{enumerate}
then $\{f_n\}_{n\in I}$ is a frame.
\end{cor}
\begin{proof}
$(1)$: 

This is clear from statement $(4)$ of Theorem \ref{doesframerecprop}.

$(2)$:

By statement $(1)$ of Theorem \ref{doesframerecprop}, $B$ is invertible and therefore, for all $f\in \Hcal$,
$\sum_{n=0}^{\infty}\langle f,f_{n}\rangle f_{n}$ or $\sum_{n=-\infty}^{\infty}\langle f,f_{n}\rangle f_{n}$ exists. Taking inner products of these terms with $f$ gets that $\{f_n\}_{n\in I}$ is a Bessel sequence by the principle of uniform boundedness. Therefore, we conclude $\{f_n\}_{n\in I}$ is a frame by $(1)$.

$(3):$

Consider the following for any $f\in \Hcal$,
$$\langle f, B^{*}f\rangle =\sum_{n\in I}|\langle f, Bf_{n}\rangle|^{2}\leq \|Bf\| \|f\|.$$  Now using $(3)$ and rearranging, we get
$$A\|f\|\leq \|Bf\|.$$ In particular, the range of $B$ is closed. Therefore, we conclude $\{f_n\}_{n\in I}$ is a frame by $(2)$.

$(4):$

We have for all $f\in \Hcal$, $$\sum_{n\in I}\|\langle f, f_n \rangle B f_n\|^{2}<\infty$$ from unconditional convergence. Also, \(\inf_n \|B f_n\| > 0\) get us that $\{f_{n}\}$ is a Bessel sequence by the principle of uniform boundedness; therefore, we conclude $\{f_n\}_{n\in I}$ is a frame by $(1)$.
\end{proof}

Now using a frame operator, we give a classification of sequences that do frame reconstruction. This result is inspired by the work of Christensen in \cite{Christensen1995Frames} where the author utilizes the bounded right inverse of a synthesis operator.
\begin{thm}
Let $\Hcal$ be a Hilbert space and suppose $\{f_{n}\}_{n\in I}\subseteq \Hcal$ where $I$ is $\mathbb{N}$ or $\mathbb{Z}$. Then $\{f_{n}\}_{n\in I}\subseteq \Hcal$ does frame reconstruction if and only if there is a dense subspace $D\subseteq \Hcal$ such that the map $S:D\to \Hcal$ where
$$Sf= \sum_{n\in I}\langle f, f_{n}\rangle f_{n}$$ is defined, closed, and surjective, and the limit is taken in the appropriate way, depending on $I$.\end{thm}
\begin{proof}
Suppose there is a dense subspace $D\subseteq \Hcal$ such that the map $S:D\to \Hcal$ is defined, closed, and surjective. Then there is a $B\in B(\Hcal)$ such that for all $f\in \Hcal$,
$$SBf=f,$$ so $\{f_{n}\}_{n\in I}$ does frame reconstruction with $B^{*}$.

Now suppose that $\{f_{n}\}_{n\in I}\subseteq \Hcal$ does frame reconstruction with some operator $B$, which we now know must be self-adjoint by Theorem \ref{doesframerecprop}. Then S is defined on $ImB$, which we now know must be dense in $\Hcal$ by Theorem \ref{doesframerecprop}. Also, note that $S$ is surjective on $ImB$. To show that $S$ on $ImB$ is a closed operator, suppose that $Bf_n\to g$ for some sequence $\{f_{n}\}\subseteq \Hcal$ and $g\in \Hcal$ and that $SBf_n=f_n\to f=SBf$ for some $f\in \Hcal$. Then by continuity of $B$, $g=Bf.$
\end{proof}

\begin{rem}
A similar argument shows that $\{f_{n}\}_{n\in I}\subseteq \Hcal$ does unconditional frame reconstruction if and only if there is a dense subspace $D\subseteq \Hcal$ such that the map $S:D\to \Hcal$ where
$$Sf= \sum_{n\in I}\langle f, f_{n}\rangle f_{n}$$ is defined, closed, and surjective, where the sum converges unconditionally.
\end{rem}

\section{Two sequences that don't do frame reconstruction}\label{s4}
In this section, we provide two examples of sequences that possess frame-like properties but fail to do frame reconstruction. Our first example arises from the Kaczmarz algorithm. 

The Kaczmarz algorithm in Hilbert spaces is an alternate method to address the problem of reconstructing any vector \( f \in \mathcal{H} \) from some inner product data $\{\langle f, f_n \rangle \}$ in a stable manner, see \cite{Haller2005Kaczmarz},\cite{Kwapien2001Kaczmarz}. However, as we now show, a sequence \( \{f_n\} \) being \emph{effective} for the Kaczmarz algorithm—that is, enabling stable recovery through this iterative process—does \emph{not} necessarily imply that \( \{f_n\} \) does frame reconstruction. Additionally, it is known that if \( \{f_n\} \) is a frame, this also does not necessarily imply that \( \{f_n\} \) is effective for the Kaczmarz algorithm, see \cite{Czaja2013Kaczmarz} for details. Therefore, the frame reconstruction approach and the Kaczmarz algorithm are incomparable methods for addressing the stable signal recovery problem.

The following was shown by Herr and Weber \cite{Herr2017Fourier} via Kwapien and Mycielski's result \cite{Kwapien2001Kaczmarz} using the Kaczmarz algorithm:

\begin{thm*}[Herr and Weber]
Let $\mu$ be a singular Borel probability measure on $[0,1)$. Then there exists a Parseval frame $\{g_{n}\}_{n=0}^{\infty}\subseteq L^{2}(\mu)$ such that
$$f=\sum_{n=0}^{\infty}\langle f, g_{n}\rangle e^{2\pi i nx}$$ for all $f\in L^{2}(\mu)$.
\end{thm*}
We show that if $\{e^{2\pi i n x}\}_{n=0}^\infty \subseteq L^2(\mu)$ does frame reconstruction for some singular Borel probability measure on $[0,1)$, then $\mu$ is a Rajchman measure. Rajchman measures have a rich history that the reader can see here, \cite{Lyons1995Seventy}.
\begin{lem}
Let $\mu$ be a singular Borel probability measure on $[0,1)$. Suppose that
$
\{e^{2\pi i n x}\}_{n=0}^\infty \subseteq L^2(\mu)
$
does frame reconstruction. Then $$\langle 1, e^{2\pi i n x} \rangle \to 0.$$
\end{lem}
\begin{proof}
By Theorem \ref{doesframerecprop}, the operator
$
f \mapsto \{\langle f, e^{2\pi i n x} \rangle\}_{n=0}^\infty
$
is densely defined as a map from $L^2(\mu)$ into $\ell^2(\mathbb{N})$.
In particular, for every $\epsilon > 0$, there exists $f_{\epsilon} \in L^2(\mu)$ such that
$$
\|1 - f_{\epsilon}\| < \epsilon \quad$$and $ \{\langle f_{\epsilon}, e^{2\pi i n x} \rangle\}_{n=0}^\infty \in \ell^2(\mathbb{N}).
$

It follows that for all $n \in \mathbb{N}$,
\[
|\langle 1, e^{2\pi i n x} \rangle| < \epsilon + |\langle f_{\epsilon}, e^{2\pi i n x} \rangle|.
\]

Therefore,
\[
\limsup_{n \to \infty} |\langle 1, e^{2\pi i n x} \rangle| \leq \epsilon.
\]

Since $\epsilon > 0$ was arbitrary, we conclude that
\[
\langle 1, e^{2\pi i n x} \rangle \to 0.
\]
\end{proof}

In fact, the previous proof showed that $\{e^{2\pi i n x}\}_{n=0}^\infty \subseteq L^2(\mu)$ when $\mu$ is the Cantor-Lebesgue middle-thirds measure fails to have even a simple frame property discussed in Theorem \ref{doesframerecprop}, even though it is effective with the Kaczmarz algorithm.
\begin{cor}
Let $\mu$ be the Cantor-Lebesgue middle-thirds measure on $[0,1)$. Then there is no injective $B\in B(L^{2}(\mu))$ so that $\{Be^{2\pi i nx}\}_{n=0}^{\infty}\subseteq L^2(\mu)$ is a Bessel sequence.
\end{cor}
The reader can see \cite{Jorgensen1998Dense} for Fourier coefficients of the Cantor-Lebesgue middle-thirds measure.
In particular, when $\mu$ is the Cantor-Lebesgue middle-thirds measure on $[0,1)$, the frame operator corresponding to $
\{e^{2\pi i n x}\}_{n=0}^\infty \subseteq L^2(\mu)$ is not densely defined.

Now we show that $\{f_{n}\}_{n\in \mathbb{Z}}$ having the property that $\{Bf_n\}_{n\in \mathbb{Z}}$ is a Parseval frame for some $B\in B(\mathcal{H})$ that is positive semi-definite is not a sufficient condition for $\{f_{n}\}_{n\in \mathbb{Z}}$ doing frame reconstruction even though it is necessary.

\begin{thm}\label{abscontfail}
There exists a measure $\mu$ on $[0,1)$ that is absolutely continuous with respect to Lebesgue measure such that there is an $f\in L^{2}(\mu)$ where no $\{c_{n}\}\subseteq \mathbb{C}$ has 
$$f=\lim \sum_{n=-M}^{M}c_{n}e^{2\pi inx}$$ with convergence in $L^{2}(\mu)$. However, there is positive semi-definite $B\in B(L^{2}(\mu))$ where $\{Be^{2\pi i nx}\}_{n\in \mathbb{Z}}\subseteq L^{2}(\mu)$ is an orthonormal basis.
\end{thm}
\begin{proof}
Let $f$ be a bounded continuous function on $[0,1)$ such that
$$\limsup_{M\to \infty}|S_{M}(f)(0)|=\infty$$ where $S_{M}(f)$ denotes the classic Fourier partial sums of $f$. There is a $g\in L^{1}([0,1))$ such that the sequence
$\int_{0}^{1}|S_{M}(f)|^{2}gdx$ is unbounded. Otherwise by principle of uniform boundedness, we would have
$$sup_{n}\|S_{M}(f)\|_{\infty}<\infty.$$

Let $\mu$ be absolutely continuous with respect to Lebesgue measure with Radon-Nikodym derivative $|g|+1$. Note that $\{Be^{2\pi i nx}\}_{n\in \mathbb{Z}}$ is an orthonormal basis where $B\in B(L^{2}(\mu))$ is multiplication by $\frac{1}{\sqrt{1+|g|}}\in L^{\infty}(\mu)$.

However, suppose for the sake of contradiction that for some $\{c_{n}\}\subseteq \mathbb{C}$, $$f=\lim_{M\to \infty} \sum_{n=-M}^{M}c_{n}e^{2\pi i nx}\implies \frac{f}{\sqrt{1+|g|}}=\lim_{M\to \infty} \sum_{n=-M}^{M}c_{n}\frac{e^{2\pi i nx}}{\sqrt{1+|g|}}$$
$$\implies c_{n}=\langle \frac{f}{\sqrt{1+|g|}}, \frac{e^{2\pi inx}}{\sqrt{1+|g|}}\rangle_{\mu}=\langle f, e^{2\pi inx}\rangle_{L^{2}([0,1))}$$ for all $n$ by uniqueness of coefficients from an orthonormal basis.
Then it must be that $\int_{0}^{1}|S_{M}(f)|^{2}(|g|+1)dx$ is a bounded sequence, which is a contradiction.
\end{proof}

In this example, the frame operator associated with $\{e^{2\pi i nx}\}_{n\in \mathbb{Z}}$ is not surjective on $ImB$, but it is densely defined on $Im B$.

\section{A Paley–Wiener type stability theorem for bounded sequences that do frame reconstruction}\label{s5}
The following result is a Paley–Wiener type stability theorem for bounded sequences that do frame reconstruction, which was inspired by the Paley–Wiener type results from Christensen and Zakowicz \cite{Christensen2017Paley-Wiener},\cite{Christensen1995Paley-Wiener}. It shows that if a bounded sequence does frame reconstruction, then sufficiently small $\ell^1$ perturbations preserve this reconstruction property, even if the original sequence is not a frame. In fact, the non-frame property is also stable under such perturbations.
\begin{thm}
Let $\Hcal$ be a non-trivial Hilbert space and suppose $\{f_{n}\}_{n\in I}\subseteq \Hcal$ where $I$ is $\mathbb{N}$ or $\mathbb{Z}$ does (unconditional) frame reconstruction with $B$ and
\[
M := \sup_n \|f_n\| < \infty.
\]
Suppose further \(\{h_n\}_{n\in I}\subseteq \mathcal{H}\) satisfies
\[
\sum_{n\in I} \|f_n - h_n\| \leq \frac{1}{2}\left(\sqrt{M^2 + \frac{1}{\|B\|}} - M \right).
\]
Then \(\{h_n\}\) also does (unconditional) frame reconstruction with operator \newline \((T^{-1})^{*} B\), where
$
T : \mathcal{H} \to \mathcal{H}$ $$ \quad T f := \sum_{n\in I} \langle f, Bh_n \rangle h_n,
$$
and the sum is on $\mathbb{N}$ or $\mathbb{Z}$ in the appropriate way.

Moreover, if \(\{f_n\}\) is \emph{not} a frame, then \(\{h_n\}\) is also \emph{not} a frame.
\end{thm}

\begin{proof}
We will assume that $I=\mathbb{N}$ for convenience. The proof for when $I=\mathbb{Z}$ or when the convergence is unconditional is similar.
First, we show that for all $f\in \Hcal$, the sequence $\sum_{n=0}^{M}\langle f, Bh_{n}\rangle h_{n}$ is Cauchy.
For $f\in \Hcal$ and $k,m$ be large, consider
$$\left \|\sum_{n=k}^{m}\langle f, Bh_n\rangle h_n \right\| \leq$$
$$\left \|\sum_{n=k}^{m}\langle f, Bh_n\rangle (h_n-f_n) \right\|+\left\|\sum_{n=k}^{m}\langle f, B(h_n-f_n)\rangle f_n \right\|+ \left \|\sum_{n=k}^{m}\langle f, Bf_n\rangle f_n \right\|.$$
Now it is easy to see that the sequence $\sum_{n=0}^{M}\langle f, Bh_{n}\rangle h_{n}$ is Cauchy by boundedness of $\{f_n\}$ and $\{h_{n}\}$, the convergence of sequence $\sum_{n=0}^{M}\|f_{n}-h_{n}\|$, and fact that the sequence $\sum_{n=0}^{M}\langle f, Bf_{n}\rangle f_{n}$ is Cauchy. Therefore, $T$ is bounded by the principle of uniform boundedness.

Now let $\lambda=\frac{1}{2}\left(\sqrt{M^2 + \frac{1}{\|B\|}} - M \right)$. Consider the following for any $f\in \Hcal$,
$$\left\|\sum_{n=0}^{\infty}\langle f, Bh_{n}\rangle h_{n}-\sum_{n=0}^{\infty}\langle f, Bf_{n}\rangle f_{n} \right \| \leq$$
$$\left \|\sum_{n=0}^{\infty}\langle f, Bh_n\rangle (h_n-f_n)\right \|+ \left\|\sum_{n=0}^{\infty}\langle f, B(h_n-f_n)\rangle f_n\right\| \leq$$ $$\lambda \|f\| \|B\| sup_{n}\|h_n\|+ \|B\| \|f\| \lambda M \leq \lambda \|f\| \|B\| (M+\lambda)+\|B\| \|f\| \lambda M=$$
$$\lambda \|B\|((M+\lambda)+M)\|f\|.$$
Therefore, if $\lambda$ is chosen so that $$\lambda(2M+\lambda)\|B\|<1, $$ $T$ is invertible since 
$$\sum_{n=0}^{\infty}\langle f, Bf_{n}\rangle f_{n}=f$$ for all $f\in \Hcal$.

Now if \(\{f_n\}\) is not a frame, by Corollary \ref{almostframecor}, there exists \(f \in \mathcal{H}\) such that
\[
\sum_{n=0}^\infty |\langle f, f_n \rangle|^2 = \infty.
\]
Using the inequality
$$
|\langle f, h_n \rangle|^2 \geq (|\langle f,f_n\rangle|-|\langle f,f_n-h_n\rangle|)^{2}
\geq  |\langle f, f_n \rangle|^2 - 2|\langle f, f_n \rangle|\|f\| \|f_n-h_n\| ,
$$
and noting that \(\sum_{n\in I} \|f_n - h_n\| < \infty\) while sequence \(|\langle f, f_n \rangle|\) is bounded, it follows that
\[
\sum_{n\in I} |\langle f, h_n \rangle|^2 = \infty.
\]
Hence, \(\{h_n\}\) cannot be a frame.
\end{proof}

\begin{cor}
Let $\{f_{n}\}\subseteq \ell^{2}(\mathbb{N})$ be from Example \ref{framereconex2}. For any $\{h_{n}\}\subseteq \ell^{2}(\mathbb{N})$ such that
$$\sum_{n=0}^{\infty}\|f_n-h_n\|\leq \frac{\sqrt{2}-1}{2},$$ $\{h_{n}\}$ is a bounded sequence that does unconditional frame reconstruction and is not a frame.
\end{cor}

\section{A special weighted frame classification for sequences that do frame reconstruction}\label{s6}

In this section, we connect sequences that do frame reconstruction with weighted frames. Applying the conjecture in \cite{Balazs2023Weighted}, which is now addressed with a proof by Tselishchev in \cite{Tselishchev2025Rescaling}, to our setting, Balazs, Corso, and Stoeva, would say that if $\{f_n\}$ does unconditional frame reconstruction, then $\{f_n\}$ is a weighted frame. While were unable to prove that sequences that do unconditional frame reconstruction are weighted frames, we will provide necessary and sufficient conditions for which norm bounded sequences $\{f_n\}$ that do frame reconstruction (or sequences that do unconditional frame reconstruction) yield $\left\{\frac{f_n}{\|f_n\|} \right\}$ as a frame. We do this by utilizing a result of Yu in \cite{Yu2024Frame}:
\begin{rem}
Yu proves the following stronger statement in \cite{Yu2024Frame} even though it is not stated this way:
\end{rem}

\begin{thm*}[Yu]
Let $\Hcal$ be a Hilbert space and suppose $\left\{\frac{f_n}{\|f_n\|} \right\}_{n\in I}\subseteq \Hcal$ is a Bessel sequence. Then $\{f_{n}\}_{n\in I}$
is a finite union of unconditional Schauder sequences.
\end{thm*}

\begin{thm}
Let $\Hcal$ be a Hilbert space and suppose  $\{f_{n}\}_{n\in I}\subseteq \Hcal$ where $I$ is $\mathbb{N}$ or $\mathbb{Z}$ consists of no zero vectors. If one of the following holds:
\begin{enumerate}
    \item $\{f_{n}\}_{n\in I}$ does frame reconstruction with $B$ and \[
M := \sup_n \|f_n\| < \infty.
\]
\item $\{f_{n}\}_{n\in I}$ does unconditional frame reconstruction with $B$.
\end{enumerate}
 Then $\left\{\frac{f_n}{\|f_n\|} \right\}_{n\in I}$ and $\{\|f_n\|Bf_n\}_{n\in I}$ are dual frames if and only if $\{f_{n}\}_{n\in I}$ is a finite union of sequences that are unconditional Schauder bases for their closed spans.
\end{thm}
\begin{proof}
Suppose that $\left\{\frac{f_n}{\|f_n\|} \right\}_{n\in I}$ and $\{\|f_n\|Bf_n\}_{n\in I}$ are dual frames. Then the result follows by the theorem of Yu.

Suppose that $\{f_{n}\}_{n\in I}$ is a finite union of sequences that are unconditional Schauder bases for their closed spans. Then $\left\{\frac{f_n}{\|f_n\|} \right\}_{n\in I}$ is a finite union of sequences that are Riesz bases for their closed spans. In particular, $\left\{\frac{f_n}{\|f_n\|} \right\}_{n\in I}$ is a Bessel sequence. We conclude by arguing that $\{\|f_n\|Bf_n\}_{n\in I}$ is a Bessel sequence. If condition $(1)$ holds, then the result follows from $\{Bf_n\}_{n\in I}$ being a Bessel sequence by Theorem \ref{doesframerecprop}. If condition $(2)$ holds, then for all $f\in \Hcal$,
$$\sum_{n\in I}|\langle f, Bf_n\rangle |^{2}\|f_n\|^{2}<\infty,$$ and the result follows.
\end{proof}

\section{Classification of unconditional Schauder bases that do frame reconstruction}\label{s7}

In this section, we classify all unconditional Schauder bases that do frame reconstruction. We do this by utilizing the following result in \cite{Christensen2016Frames}:

\begin{thm*}[Christensen, Lemma 3.3.3]
Let $\Hcal$ be a Hilbert space and suppose  $\{f_{n}\}_{n=0}^{\infty}\subseteq \Hcal$ is a Schauder basis with associated biorthogonal system $\{g_{k}\}$ where $\{g_{k}\}$ is a Bessel sequence with bound $\mathcal{B}$. Then
$$\frac{1}{\mathcal{B}}\sum_{n=0}^{\infty} |c_{n}|^{2}\leq \left \|\sum_{n=0}^{\infty} c_n f_n \right \|^{2}$$
for all finite sequences $\{c_n\}_{n=0}^{\infty}$.
\end{thm*}

\begin{thm}
Let $\Hcal$ be a Hilbert space and suppose  $\{f_{n}\}_{n\in I}\subseteq \Hcal$ is an unconditional Schauder basis. Then $\{f_{n}\}_{n\in I}$ does frame reconstruction if and only if $$ \inf_{n}
\|f_n\|>0.$$
\end{thm}
\begin{proof}
Note, $\left\{\frac{f_n}{\|f_n\|} \right\}_{n\in I}$ is a Riesz basis with dual frame $\left\{S^{-1}\frac{f_n}{\|f_n\|} \right\}_{n\in I}$ where $S$ denotes the frame operator corresponding to $\left\{\frac{f_n}{\|f_n\|} \right\}_{n\in I}$. 

Suppose that $\{f_{n}\}_{n\in I}$ does frame reconstruction with operator $B$. It easily follows that $$\langle f_n, Bf_k\rangle=\delta_{nk}$$ by the basis property. Furthermore, since $\{Bf_n\}$ is a Bessel sequence by Theorem \ref{doesframerecprop}, the previous theorem applies and shows that $$ \inf_n
\|f_n\|>0.$$

Now suppose that $$ \inf_n
\|f_n\|>0.$$ Now since $\left\{S^{-1}\frac{f_n}{\|f_n\|} \right\}_{n\in I}$ is a Riesz basis, it follows that $\left\{S^{-1}\frac{f_n}{\|f_n\|^{2}} \right\}$ is a Bessel sequence and
$$\left\langle f_n, S^{-1}\frac{f_k}{\|f_k\|^{2}} \right \rangle=\delta_{nk},$$
and the previous theorem applies.
We conclude the proof by showing that the linear map defined on $span\{f_n\}$ given by $f_n\to S^{-1}\frac{f_n}{\|f_n\|^{2}}$ is bounded. By the previous theorem and since $\left\{S^{-1}\frac{f_n}{\|f_n\|^{2}} \right\}$ is a Bessel sequence with bound $\mathcal{B}$,
$$\frac{1}{\mathcal{B}}\left\|\sum_{n=0}^{\infty}c_nS^{-1}\frac{f_n}{\|f_n\|^{2}} \right\|^{2}\leq \sum_{n=0}^{\infty} |c_{n}|^{2}\leq \mathcal{B}\left \|\sum_{n=0}^{\infty} c_n f_n \right \|^{2}$$
for all finite sequences $\{c_n\}_{n=0}^{\infty}$, and the result follows.
\end{proof}

\begin{rem}
In other words, an $\omega-$linearly independent sequence does unconditional frame reconstruction
if and only if its a norm bounded below unconditional Schauder basis.
\end{rem}

\begin{rem}
Note that an $\omega-$linearly independent frame must be a Riesz basis, but an $\omega-$linearly independent sequence that does frame reconstruction need not be a Riesz basis as we have seen in Example \ref{firstnonframeex}.
\end{rem}

\begin{cor}
A norm bounded unconditional Schauder basis does frame reconstruction if and only if its a Riesz basis.
\end{cor}

\section{Classification of Exponential Schauder bases that do frame reconstruction}\label{s8}
We conclude by classifying all finite Borel measures $\mu$ on $[0,1)$ that have the property that $\{ e^{2 \pi i n x} \}_{n \in \mathbb{Z}} \subseteq L^2(\mu)$ is a Schauder basis that does frame reconstruction. The results of this section are motivated by the study of operator orbit frames in \cite{Berner2024Operator}. The conclusions of the following Lemma and Corollary align with those obtained there, now within the context of sequences that do frame reconstruction. Furthermore, the following theorem in \cite{Berner2024Frame-like} will help with the classification, and its proof was inspired from the work of Lai \cite{Lai2011Fourier}:
\begin{thm*}[Berner]
Let $\mu$ be a finite Borel measure on $[0,1)$, and let $g$ denote the Radon–Nikodym derivative of its absolutely continuous part with respect to Lebesgue measure. Suppose that there is a Bessel sequence $\{g_{n}\}$ with bound $B$ such that
$$\lim_{M\to \infty}\sum_{n=-M}^{M}\langle f, g_{n}\rangle e^{2\pi i nx}=f$$
for all $f\in L^{2}(\mu)$ with convergence in norm, then $$g(x)>\frac{1}{3B}$$ Lebesgue almost everywhere on its support.
\end{thm*}

We also utilize the celebrated result about what weights for Lebesgue measure allow convergence of classical Fourier partial sums \cite{Hunt1973Weighted}:
\begin{thm*}[Hunt, Muckenhoupt, and Wheeden]\label{Mthm}
Let $w\in L^{1}([0,1))$. The following are equivalent:
\begin{enumerate}
    \item $w$ satisfies the $A_2$ condition.
    \item If $\int_{0}^{1}|f|^{2}wdx<\infty$, then
    $$\lim_{M\to \infty}\int_{0}^{1}|S_M(f)-f|^{2}wdx=0$$ where $S_{M}$ denotes the classic Fourier partial sum of $f$.
\end{enumerate}
\end{thm*}
For the definition of the $A_2$ condition, the reader can check out \cite{Hunt1973Weighted}.
\begin{lem}\label{almosta2}
Suppose $\mu$ is a finite Borel measure on $[0,1)$, and 

$\{ e^{2 \pi i n x} \}_{n \in \mathbb{Z}}\subseteq L^2(\mu)$ does frame reconstruction with $B$. Then:
\begin{enumerate}
    \item $\mu$ is absolutely continuous with respect to Lebesgue measure.
    \item If $g = \frac{d\mu}{dx}$, then
    \[
    B f(x) = \frac{1}{g(x)} f(x).
    \]
\end{enumerate}
\end{lem}
\begin{proof}
$(1)$:

Let $$\mu = \mu_a + \mu_s$$ be the Lebesgue decomposition of the finite Borel measure \(\mu\) on \([0,1)\), where \(\mu_a \ll dx\) and \(\mu_s \perp dx\). There exists a Borel set \(S \subset [0,1)\) such that \(S\) has Lebesgue measure zero and
\[
\mu_s(S^c) = 0.
\]
Since $B$ is self-adjoint and by the frame reconstruction property,
\begin{equation}\label{acequation}
\langle \chi_S, B\chi_S \rangle = \sum_{n=-\infty}^\infty \left| \langle \chi_S, B e^{2 \pi i n x} \rangle \right|^2 = \sum_{n=-\infty}^\infty \left| \langle B \chi_S, e^{2 \pi i n x} \rangle \right|^2<\infty.
\end{equation}

Therefore, the complex measure $B\chi_{S}d\mu$ is absolutely continuous with respect to Lebesgue measure, showing $B\chi_{S}$ is zero $\mu_{s}$ almost everywhere. By equation \eqref{acequation} and since $\{Be^{2\pi i nx}\}_{n=0}^{\infty}$ is complete, 
$$\chi_{S}=0.$$

$(2)$:
Let $k\in \mathbb{Z}$ and consider by a similar calculation,
$$\langle e^{2\pi ikx},Be^{2\pi i kx}\rangle=\sum_{n=-\infty}^{\infty}|\langle Be^{2\pi ikx}, e^{2\pi i nx}\rangle|^{2}<\infty. $$
Therefore, $$\int_{0}^{1}|Be^{2\pi ikx}g|^{2}dx<\infty.$$ Now by Carleson's theorem, we have
$$\lim_{M\to \infty}\sum_{n=-M}^{M}\langle Be^{2\pi ikx}, e^{2\pi i nx}\rangle e^{2\pi i nx}\to (Be^{2\pi ikx})g$$ point-wise Lebesgue almost everywhere and also $\mu$ almost everywhere since $\mu$ is absolutely continuous with respect to Lebesgue measure. However, by the frame reconstruction condition, there is a subsequence of the above sequence that converges $\mu$-almost everywhere to $e^{2\pi i kx}$.
Now we can conclude that \begin{equation}\label{1/geq}
(Be^{2\pi ikx})g=e^{2\pi ikx},\end{equation} for $\mu$ almost every $x$. Furthermore, by the frame reconstruction property, $g$ is bounded below Lebesgue almost everywhere on its support; therefore, multiplication by $\frac{1}{g}$ defines a bounded operator on $L^{2}(\mu)$, which proves $(2)$ by completeness of $\{e^{2\pi i nx}\}_{n\in \mathbb{Z}}\subseteq L^{2}(\mu)$ and equation \eqref{1/geq}.

\end{proof}
The following is a consequence of Corollary \ref{almostframecor}:
\begin{cor}
Let $\mu$ be a finite Borel measure on $[0,1)$ and suppose $\{ e^{2 \pi i n x} \}_{n \in \mathbb{Z}} \subseteq L^2(\mu)$ does frame reconstruction with $g=\frac{d\mu}{dx}$. Then the following are equivalent:
\begin{enumerate}
    \item $\{ e^{2 \pi i n x} \}_{n \in \mathbb{Z}} \subseteq L^2(\mu)$ is a Bessel sequence.
    \item $\left\{\frac{e^{2\pi i nx}}{g} \right\}_{n\in \mathbb{Z}}\subseteq L^2(\mu)$ is a lower-semi frame.
    \item $\{ e^{2 \pi i n x} \}_{n \in \mathbb{Z}} \subseteq L^2(\mu)$ does unconditional frame reconstruction.
    \item $\{ e^{2 \pi i n x} \}_{n \in \mathbb{Z}} \subseteq L^2(\mu)$ is a frame.
    \item $\left\{\frac{e^{2\pi i nx}}{g} \right\}_{n\in \mathbb{Z}}\subseteq L^2(\mu)$ is a frame.
    \item $g$ is bounded.
\end{enumerate}
\end{cor}

What follows is our classification of finite Borel measures $\mu$ on $[0,1)$ that have the property that $\{ e^{2 \pi i n x} \}_{n \in \mathbb{Z}} \subseteq L^2(\mu)$ is a Schauder basis that does frame reconstruction:
\begin{thm}
Let $\mu$ be a finite Borel measure on $[0,1)$. Then $\{e^{2\pi i n x}\}_{n\in\mathbb{Z}} \subseteq L^{2}(\mu)$ is a Schauder basis that does frame reconstruction if and only if $\mu$ is absolutely continuous with respect to Lebesgue measure, with density $g = \tfrac{d\mu}{dx}$ such that $g$ is bounded below and satisfies the $A_{2}$ condition.

Furthermore, if the above condition is satisfied, the following are equivalent:
\begin{enumerate}
    \item $\{ e^{2 \pi i n x} \}_{n \in \mathbb{Z}} \subseteq L^2(\mu)$ is a Bessel sequence.
    \item $\left\{\frac{e^{2\pi i nx}}{g} \right\}_{n\in \mathbb{Z}}\subseteq L^2(\mu)$ is a lower-semi frame.
    \item $\{ e^{2 \pi i n x} \}_{n \in \mathbb{Z}} \subseteq L^2(\mu)$ is an unconditional Schauder basis.
     \item $\left\{\frac{e^{2\pi i nx}}{g} \right\}_{n\in \mathbb{Z}}\subseteq L^2(\mu)$ is an unconditional Schauder basis.
    \item $\{ e^{2 \pi i n x} \}_{n \in \mathbb{Z}} \subseteq L^2(\mu)$ is a Riesz basis.
    \item $\left\{\frac{e^{2\pi i nx}}{g} \right\}_{n\in \mathbb{Z}}\subseteq L^2(\mu)$ is a Riesz basis.
    \item $g$ is bounded.
\end{enumerate}

\end{thm}
\begin{proof}
Suppose that $\{ e^{2 \pi i n x} \}_{n \in \mathbb{Z}}$ is a Schauder basis that does frame reconstruction with $B$. Then it is easy to see that 
$$\langle e^{2 \pi i n x},  Be^{2 \pi i k x}\rangle=\delta_{nk}.$$ Now note that for each $k\in \mathbb{Z}$, the complex measure $Be^{2\pi i kx}d\mu$ has the same Fourier coefficients as the complex measure $e^{2\pi ikx}dx$. Therefore by Lemma \ref{almosta2},
$$g(x)Be^{2\pi ikx}=e^{2\pi ikx}$$ for Lebesgue almost every $x$ where $g = \tfrac{d\mu}{dx}$. In particular, $$g(x)>0$$ for Lebesgue almost every $x$. It follows, that $$\lim_{M\to \infty} S_{M}(f)=f$$ for all $f\in L^{2}(\mu)$ where $S_{M}$ denotes the classic Fourier partial sums of $f$. Therefore, $g$ satisfies the $A_{2}$ condition. Also, $g$ is bounded below by the frame reconstruction property and since $g(x)>0$ for Lebesgue almost every $x$.

Now suppose that $\mu$ is absolutely continuous with respect to Lebesgue measure, with density $g = \tfrac{d\mu}{dx}$ such that $g$ is bounded below and satisfies the $A_{2}$ condition. It follows that $$\lim_{M\to \infty} S_{M}(f)=f$$ for all $f\in L^{2}(\mu)$. Therefore, $\{ e^{2 \pi i n x} \}_{n \in \mathbb{Z}} \subseteq L^2(\mu)$ does frame reconstruction with $B$ where $B$ is multiplication by $\frac{1}{g}$. Finally, it is easy to check that $$\langle e^{2 \pi i n x},  Be^{2 \pi i k x}\rangle=\delta_{nk},$$ which shows that $\{ e^{2 \pi i n x} \}_{n \in \mathbb{Z}} \subseteq L^2(\mu)$ is a Schauder basis.

Finally, under the condition, statements 1-5 follow from Corollary \ref{almostframecor}.
\end{proof}

\section*{Acknowledgments}
I thank Dr.\ Eric S.\ Weber for his ideas and insights. This work is partly based on our collaborations, and my classification of sequences that do frame reconstruction was inspired by his idea of a densely defined frame operator.

\printbibliography

@article {Hunt1973Weighted,
    AUTHOR = {Hunt, Richard and Muckenhoupt, Benjamin and Wheeden, Richard},
     TITLE = {Weighted norm inequalities for the conjugate function and
              {H}ilbert transform},
   JOURNAL = {Trans. Amer. Math. Soc.},
  FJOURNAL = {Transactions of the American Mathematical Society},
    VOLUME = {176},
      YEAR = {1973},
     PAGES = {227--251},
      ISSN = {0002-9947,1088-6850},
   MRCLASS = {42A40 (44A15 47G05)},
  MRNUMBER = {312139},
MRREVIEWER = {P.\ G.\ Rooney},
       DOI = {10.2307/1996205},
       URL = {https://doi.org/10.2307/1996205},
}

@article {Berner2024Frame-like,
    AUTHOR = {Berner, Chad},
     TITLE = {Frame-like {F}ourier expansions for finite {B}orel measures on
              {$\Bbb R$}},
   JOURNAL = {Pure Appl. Funct. Anal.},
  FJOURNAL = {Pure and Applied Functional Analysis},
    VOLUME = {9},
      YEAR = {2024},
    NUMBER = {6},
     PAGES = {1527--1545},
      ISSN = {2189-3756,2189-3764},
   MRCLASS = {42A16 (42A20 42C15 46C07)},
  MRNUMBER = {4853157},
MRREVIEWER = {Beata\ Deregowska},
}

@article {Lai2011Fourier,
    AUTHOR = {Lai, Chun-Kit},
     TITLE = {On {F}ourier frame of absolutely continuous measures},
   JOURNAL = {J. Funct. Anal.},
  FJOURNAL = {Journal of Functional Analysis},
    VOLUME = {261},
      YEAR = {2011},
    NUMBER = {10},
     PAGES = {2877--2889},
      ISSN = {0022-1236,1096-0783},
   MRCLASS = {42C15 (28A80 46B15)},
  MRNUMBER = {2832585},
MRREVIEWER = {Peter\ R.\ Massopust},
       DOI = {10.1016/j.jfa.2011.07.014},
       URL = {https://doi.org/10.1016/j.jfa.2011.07.014},
}

@article{Kwapien2001Kaczmarz,
author = {Stanisław Kwapień and Jan Mycielski},
journal = {Studia Mathematica},
language = {eng},
number = {1},
pages = {75-86},
title = {On the Kaczmarz algorithm of approximation in infinite-dimensional spaces},
url = {http://eudml.org/doc/284583},
volume = {148},
year = {2001},
}

@Article{Herr2017Fourier,
AUTHOR = {Herr, John E. and Weber, Eric S.},
TITLE = {Fourier Series for Singular Measures},
JOURNAL = {Axioms},
VOLUME = {6},
YEAR = {2017},
NUMBER = {2},
ARTICLE-NUMBER = {7},
URL = {https://www.mdpi.com/2075-1680/6/2/7},
ISSN = {2075-1680},
DOI = {10.3390/axioms6020007}
}

@article {Haller2005Kaczmarz,
    AUTHOR = {Haller, Rainis and Szwarc, Ryszard},
     TITLE = {Kaczmarz algorithm in {H}ilbert space},
   JOURNAL = {Studia Math.},
  FJOURNAL = {Studia Mathematica},
    VOLUME = {169},
      YEAR = {2005},
    NUMBER = {2},
     PAGES = {123--132},
      ISSN = {0039-3223,1730-6337},
   MRCLASS = {41A65 (46C99)},
  MRNUMBER = {2140451},
MRREVIEWER = {Grzegorz\ Lewicki},
       DOI = {10.4064/sm169-2-2},
       URL = {https://doi.org/10.4064/sm169-2-2},
}

@book {Gohberg1969Theory,
    AUTHOR = {Gohberg, I. C. and Kre\u in, M. G.},
     TITLE = {Introduction to the theory of linear nonselfadjoint operators},
    SERIES = {Translations of Mathematical Monographs},
    VOLUME = {Vol. 18},
      NOTE = {Translated from the Russian by A. Feinstein},
 PUBLISHER = {American Mathematical Society, Providence, RI},
      YEAR = {1969},
     PAGES = {xv+378},
   MRCLASS = {47.10},
  MRNUMBER = {246142},
}

@book {Christensen2016Frames,
    AUTHOR = {Christensen, Ole},
     TITLE = {An introduction to frames and {R}iesz bases},
    SERIES = {Applied and Numerical Harmonic Analysis},
   EDITION = {Second},
 PUBLISHER = {Birkh\"auser/Springer, [Cham]},
      YEAR = {2016},
     PAGES = {xxv+704},
      ISBN = {978-3-319-25611-5; 978-3-319-25613-9},
   MRCLASS = {42-02 (42C15 42C40 46B15 46C05)},
  MRNUMBER = {3495345},
MRREVIEWER = {Marcin\ M.\ Bownik},
       DOI = {10.1007/978-3-319-25613-9},
       URL = {https://doi.org/10.1007/978-3-319-25613-9},
}

@article {Li2001Pseudo-duals,
    AUTHOR = {Li, Shidong and Ogawa, Hidemitsu},
     TITLE = {Pseudo-duals of frames with applications},
   JOURNAL = {Appl. Comput. Harmon. Anal.},
  FJOURNAL = {Applied and Computational Harmonic Analysis. Time-Frequency
              and Time-Scale Analysis, Wavelets, Numerical Algorithms, and
              Applications},
    VOLUME = {11},
      YEAR = {2001},
    NUMBER = {2},
     PAGES = {289--304},
      ISSN = {1063-5203,1096-603X},
   MRCLASS = {46B15 (42C15 46C05)},
  MRNUMBER = {1848709},
MRREVIEWER = {Ole\ Christensen},
       DOI = {10.1006/acha.2001.0347},
       URL = {https://doi.org/10.1006/acha.2001.0347},
}

@article {Duffin1952Class,
    AUTHOR = {Duffin, R. J. and Schaeffer, A. C.},
     TITLE = {A class of nonharmonic {F}ourier series},
   JOURNAL = {Trans. Amer. Math. Soc.},
  FJOURNAL = {Transactions of the American Mathematical Society},
    VOLUME = {72},
      YEAR = {1952},
     PAGES = {341--366},
      ISSN = {0002-9947,1088-6850},
   MRCLASS = {42.4X},
  MRNUMBER = {47179},
MRREVIEWER = {J.\ Korevaar},
       DOI = {10.2307/1990760},
       URL = {https://doi.org/10.2307/1990760},
}

@article {Jorgensen1998Dense,
    AUTHOR = {Jorgensen, Palle E. T. and Pedersen, Steen},
     TITLE = {Dense analytic subspaces in fractal {$L^2$}-spaces},
   JOURNAL = {J. Anal. Math.},
  FJOURNAL = {Journal d'Analyse Math\'{e}matique},
    VOLUME = {75},
      YEAR = {1998},
     PAGES = {185--228},
      ISSN = {0021-7670,1565-8538},
   MRCLASS = {46E30 (28A75 42C05 46L55 47B38)},
  MRNUMBER = {1655831},
MRREVIEWER = {Javier\ Soria},
       DOI = {10.1007/BF02788699},
       URL = {https://doi.org/10.1007/BF02788699},
}

@article {Czaja2013Kaczmarz,
    AUTHOR = {Czaja, Wojciech and Tanis, James H.},
     TITLE = {Kaczmarz algorithm and frames},
   JOURNAL = {Int. J. Wavelets Multiresolut. Inf. Process.},
  FJOURNAL = {International Journal of Wavelets, Multiresolution and
              Information Processing},
    VOLUME = {11},
      YEAR = {2013},
    NUMBER = {5},
     PAGES = {1350036, 13},
      ISSN = {0219-6913,1793-690X},
   MRCLASS = {42C15 (41A65 65J10)},
  MRNUMBER = {3117886},
MRREVIEWER = {Aleksandr\ Krivoshein},
       DOI = {10.1142/S0219691313500367},
       URL = {https://doi.org/10.1142/S0219691313500367},
}

@misc{Berner2024Operator,
      title={Operator orbit frames and frame-like Fourier expansions}, 
      author={Berner, Weber},
      year={2024},
      eprint={2409.10706},
      archivePrefix={arXiv},
      primaryClass={math.FA}
}

@incollection {Beutler1976Operator,
    AUTHOR = {Beutler, Frederick J. and Root, William L.},
     TITLE = {The operator pseudoinverse in control and systems
              identification},
 BOOKTITLE = {Generalized inverses and applications ({P}roc. {S}em., {M}ath.
              {R}es. {C}enter, {U}niv. {W}isconsin, {M}adison, {W}is.,
              1973)},
    SERIES = {University of Wisconsin, Mathematics Research Center,
              Publication},
    VOLUME = {No. 32},
     PAGES = {397--494},
 PUBLISHER = {Academic Press, New York-London},
      YEAR = {1976},
   MRCLASS = {93E10 (65F20)},
  MRNUMBER = {490311},
MRREVIEWER = {Charles\ F.\ Van Loan},
}

@inproceedings {Lyons1995Seventy,
    AUTHOR = {Lyons, Russell},
     TITLE = {Seventy years of {R}ajchman measures},
 BOOKTITLE = {Proceedings of the {C}onference in {H}onor of {J}ean-{P}ierre
              {K}ahane ({O}rsay, 1993)},
   JOURNAL = {J. Fourier Anal. Appl.},
  FJOURNAL = {The Journal of Fourier Analysis and Applications},
      YEAR = {1995},
     PAGES = {363--377},
      ISSN = {1069-5869,1531-5851},
   MRCLASS = {42A63 (04A15 42A55 43-03 43A10 43A25 43A46)},
  MRNUMBER = {1364897},
MRREVIEWER = {Colin\ C.\ Graham},
}

@misc{Balazs2023Weighted,
      title={Weighted frames, weighted lower semi frames and unconditionally convergent multipliers}, 
      author={Peter Balazs and Rosario Corso and Diana Stoeva},
      year={2023},
      eprint={2310.18957},
      archivePrefix={arXiv},
      primaryClass={math.FA},
      url={https://arxiv.org/abs/2310.18957}, 
}

@article {Yu2024Frame,
    AUTHOR = {Yu, Pu-Ting},
     TITLE = {Frame-normalizable sequences},
   JOURNAL = {Adv. Comput. Math.},
  FJOURNAL = {Advances in Computational Mathematics},
    VOLUME = {50},
      YEAR = {2024},
    NUMBER = {4},
     PAGES = {Paper No. 89, 23},
      ISSN = {1019-7168,1572-9044},
   MRCLASS = {42C15},
  MRNUMBER = {4784098},
MRREVIEWER = {Tin\ Thien\ Tran},
       DOI = {10.1007/s10444-024-10182-z},
       URL = {https://doi.org/10.1007/s10444-024-10182-z},
}

@article {Balazs2010Weighted,
    AUTHOR = {Balazs, Peter and Antoine, Jean-Pierre and Grybo\'s, Anna},
     TITLE = {Weighted and controlled frames: mutual relationship and first
              numerical properties},
   JOURNAL = {Int. J. Wavelets Multiresolut. Inf. Process.},
  FJOURNAL = {International Journal of Wavelets, Multiresolution and
              Information Processing},
    VOLUME = {8},
      YEAR = {2010},
    NUMBER = {1},
     PAGES = {109--132},
      ISSN = {0219-6913,1793-690X},
   MRCLASS = {42C15 (42C40 65T60)},
  MRNUMBER = {2654396},
       DOI = {10.1142/S0219691310003377},
       URL = {https://doi.org/10.1142/S0219691310003377},
}

@article {Dutkay2016Weighted,
    AUTHOR = {Dutkay, Dorin Ervin and Ranasinghe, Rajitha},
     TITLE = {Weighted {F}ourier frames on fractal measures},
   JOURNAL = {J. Math. Anal. Appl.},
  FJOURNAL = {Journal of Mathematical Analysis and Applications},
    VOLUME = {444},
      YEAR = {2016},
    NUMBER = {2},
     PAGES = {1603--1625},
      ISSN = {0022-247X,1096-0813},
   MRCLASS = {42C15},
  MRNUMBER = {3535778},
MRREVIEWER = {Jean-Pierre\ Gabardo},
       DOI = {10.1016/j.jmaa.2016.07.042},
       URL = {https://doi.org/10.1016/j.jmaa.2016.07.042},
}

@article{Christensen2017Paley-Wiener,
  author    = {Ole Christensen and M. I. Zakowicz},
  title     = {Paley-Wiener Type Perturbations of Frames and the Deviation from Perfect Reconstruction},
  journal   = {Azerbaijan Journal of Mathematics},
  volume    = {7},
  number    = {1},
  year      = {2017},
  month     = {January},
  issn      = {2218-6816},
  note      = {Special issue},
}

@article {Christensen1995Frames,
    AUTHOR = {Christensen, Ole},
     TITLE = {Frames and pseudo-inverses},
   JOURNAL = {J. Math. Anal. Appl.},
  FJOURNAL = {Journal of Mathematical Analysis and Applications},
    VOLUME = {195},
      YEAR = {1995},
    NUMBER = {2},
     PAGES = {401--414},
      ISSN = {0022-247X,1096-0813},
   MRCLASS = {47A05},
  MRNUMBER = {1354551},
MRREVIEWER = {I.\ Ya.\ Novikov},
       DOI = {10.1006/jmaa.1995.1363},
       URL = {https://doi.org/10.1006/jmaa.1995.1363},
}

@article {Christensen1995Paley-Wiener,
    AUTHOR = {Christensen, Ole},
     TITLE = {A {P}aley-{W}iener theorem for frames},
   JOURNAL = {Proc. Amer. Math. Soc.},
  FJOURNAL = {Proceedings of the American Mathematical Society},
    VOLUME = {123},
      YEAR = {1995},
    NUMBER = {7},
     PAGES = {2199--2201},
      ISSN = {0002-9939,1088-6826},
   MRCLASS = {46C99 (42A99 42C15)},
  MRNUMBER = {1246520},
MRREVIEWER = {B.\ S.\ Rubin},
       DOI = {10.2307/2160957},
       URL = {https://doi.org/10.2307/2160957},
}

@misc{Tselishchev2025Rescaling,
      title={Rescaling of unconditional Schauder frames in Hilbert spaces and completely bounded maps}, 
      author={Anton Tselishchev},
      year={2025},
      eprint={2508.02802},
      archivePrefix={arXiv},
      primaryClass={math.FA},
      url={https://arxiv.org/abs/2508.02802}, 
}

\end{document}